\documentclass{amsart}
\usepackage{graphicx}
\usepackage{amsmath}
\allowdisplaybreaks[1] 
\usepackage{amssymb}
\usepackage{psfrag}

\newtheorem{thm}{Theorem}[section]

\newtheorem{lemma}[thm]{Lemma}
\newtheorem{prop}[thm]{Proposition}

\newcommand{\R}{\mathbb{R}}
\newcommand{\Z}{\mathbb{Z}}

\begin{document}
\title{Lens rigidity with trapped geodesics in two dimensions}
\author[C.~Croke]{Christopher B. Croke$^+$} \address{ Department ofbb:
Mathematics, University of Pennsylvania, Philadelphia, PA 19104-6395
USA} \email{ccroke@math.upenn.edu} \thanks{$+$Supported in part by NSF grant DMS 10-03679 and an Eisenbud Professorship at M.S.R.I.}
\author[P.~Herreros]{Pilar  Herreros$^{\dagger}$} \address{Mathematisches Institut, University of M\"unster, 48149 M\"unster, Germany} \email{p.herreros@uni-muenster.de}
 \thanks{$^\dagger$Supported in part by an M.S.R.I. Postdoctoral Fellowship.}

\keywords{Scattering rigidity, Lens rigidity, trapped geodesics}

\begin{abstract}
We consider the scattering and lens rigidity of compact surfaces with boundary that have a trapped geodesic.  In particular we show that the flat cylinder and the flat M\"obius strip are determined by their lens data.  We also see by example that the flat M\"obius strip is not determined by it's scattering data.  We then consider the case of negatively curved cylinders with convex boundary and show that they are lens rigid.
\end{abstract}

\maketitle

\section{Introduction}

In this paper we consider the lens and scattering rigidity of a number of compact surfaces with boundary that have a trapped geodesic.  A trapped geodesic ray is a geodesic $\gamma(t)$ which is defined for all $t\geq 0$, while a trapped geodesic is one defined for all $t$.  We will call a unit vector trapped if it is tangent to a geodesic ray while we call it totally trapped if it is tangent to a trapped geodesic.

We will consider compact two dimensional manifolds $(M,\partial M,g)$ with boundary $\partial M$ and metric $g$.
Let $U^+\partial M$ represent the space of inwardly pointing unit vectors at the boundary.  That is, $v\in U^+\partial M$ means that $v$ is a unit vector based at a boundary point and $\langle v,\eta^+\rangle\geq 0$, where $\eta^+$ is the unit vector of $M$ normal to $\partial M$ and pointing inward.  $U^-\partial M$ will represent the outward vectors.  These spaces are two dimensional while $U^+\partial M\cap U^-\partial M=U(\partial M)$ the unit tangent bundle of $\partial M$ is one dimensional.

For $v\in U^+\partial M$ let $\gamma_v(t)$ be the geodesic with $\gamma'(0)=v$.  We let $TT(v)\in [0,\infty]$ (the travel time) be the first time $t>0$ when $\gamma_v(t)$ hits the boundary again.  If $\gamma_v(t)$ never hits the boundary again then $TT(v)=\infty$, while if either $\gamma_v(t)$ does not exist for any $t>0$ or there are arbitrarily small values of $t>0$ such that $\gamma(t)\in \partial M$, then we let $TT(v)=0$.  Note that $TT(v)=0$ implies that $v\in U(\partial M)$ while for $v\in U(\partial M)$, $TT(v)$ may or may not be $0$.

The scattering map ${\emph{S}}:U^+\partial M\to U^-\partial M$ takes a vector $v\in U^+\partial M$ to the vector $\gamma'(TT(v))\in U^-\partial M$.  It will not be defined when $TT(v)=\infty$ and will be $v$ itself when $TT(v)=0$. If another surface $(M_1,\partial M_1,g_1)$ has isometric boundary to $(M,\partial M,g)$ in the sense that $(\partial M,g)$ ($g$ restricted to $\partial M$) is isometric to $(\partial M_1,g_1)$ (i.e. they have the same number of components - circles - with the same lengths), then we can identify $U^+\partial M_1$ with $U^+\partial M$ and $U^-\partial M_1$ with $U^-\partial M$.  We say that $(M,\partial M,g)$ and $(M_1,\partial M_1,g_1)$ have the same scattering data if they have isometric boundaries and under the identifications given by the isometry they have the same scattering map.  If in addition the travel times $TT(v)$ coincide then they are said to have the same lens data.

A compact manifold $(M,\partial M,g)$ is said to be scattering (resp. lens) rigid if for any other manifold $(M_1,\partial M_1,g_1)$ with the same scattering (resp. lens) data there is an isometry from $M_1$ to $M$ that agrees with the given isometry of the boundaries.

In this paper we prove three such rigidity results.

\begin{thm}
\label{flatcylinder}
The flat cylinder $[-1,1]\times S^1$ is lens rigid.
\end{thm}

\begin{thm}
\label{flatmobius}
The flat M\"obius strip is lens rigid.
\end{thm}

\begin{thm}\label{negative curvature}
A cylinder with negative curvature and convex boundary is lens rigid.
\end{thm}

The higher dimensional version of theorem \ref{flatcylinder} was proved recently \cite{Cr11} by the first author.  In that paper it was shown that for $n\geq 2$, $D^n\times S^1$ is scattering rigid where $D^n$ represents the unit disc in $\R^n$.  This was the first example of such a rigidity theorem that had trapped geodesic rays (however \cite{St-Uh09} has a local rigidity result that includes trapped geodesic rays).  The two dimensional case has a number of differences from the higher dimensional case.  Although it is possible to approach Theorem \ref{flatcylinder} with methods as in \cite{Cr11} there are a number of complications.  In particular, the boundary is neither connected nor does the second fundamental form have a positive eigenvalue.  Here we use a different approach entirely, which is very two dimensional and also allows us to prove the other results.   We should note that in the two dimensional case we do not prove scattering rigidity, but only lens rigidity.  We see by example (see below) that the flat M\"obius band is not scattering rigid (at least if one allows $C^1$ metrics) while the other two cases are still open.

\begin{figure}
\label{cylinders}
\includegraphics[width=50mm]{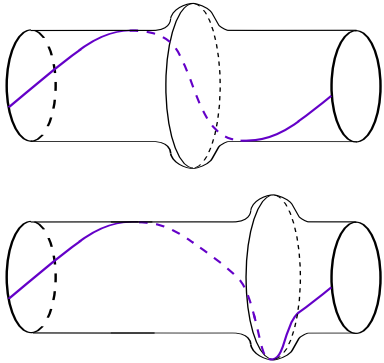}
\caption{Not isometric but same scattering and lens data.}

\end{figure}

The fact that not all manifolds are scattering rigid was pointed out in \cite{Cr91}.  For $\frac 1 4 >\epsilon>0$ let $h(t)$ be a small smooth bump function which is 0 outside $(-\epsilon ,\epsilon)$ and positive in $(-\epsilon,\epsilon)$.   For $s\in (-1+2\epsilon,1-2\epsilon)$ consider surfaces of revolution $g_s$ with smooth generating functions $F_s(t)=1+h(s+t)$ for $t\in [-1,1]$.  These surfaces of revolution look like flat cylinders with bumps on them that are shifted depending on $s$ but otherwise look the same (see figure \ref{cylinders}).  The Clairaut relations show that, independent of $s$, geodesics entering one side with a given initial condition exit out the other side after the same distance at the same point with the same angle.  Hence all metrics have the same scattering data (and in fact lens data) but are not isometric.  A much larger class of examples was given in section 6 of \cite{Cr-Kl94}.

\bigskip

We now present an example that shows that the flat M\"obius band is not scattering rigid.
Let $C$ be the cylinder $[0,l]\times S^1$ and let $H$ be a hemisphere attached to
$C$ by identifying the equator with the the curve $l\times S^1$. We get $M_1=C\cup H/\sim$ where $\sim$ is the identification above.  Note that $M_1$ is topologically a disc.

We need to understand some of the geodesics on $M_1$. Observe that any geodesic in the cylinder
that reaches $l\times S^1$ forming an angle $\alpha$ with it goes into $H$, where it is a great
circle that leaves $H$ again at its antipodal point forming the same angle $\alpha$. From the point
of view of the cylinder, any geodesic that leaves it through a point $(l,\theta)$ comes back at the
point $(l,\theta+\pi)$ with the same angle. Thus, the scattering data of $M_1$ is the same as that of $M_0$; the cylinder with one boundary identified to itself via the antipodal map. I.e. $M_0$ is a flat M\"obius band. Therefore, the scattering data of $M_0$ and $M_1$ are the same, but the travel times are different.  In fact they differ by exactly $\pi$.

\begin{figure}
\label{caps}
\psfrag{M}{\hspace*{-0.2cm}$M_0$}
\psfrag{C}{$C$}
\psfrag{Hl}{$H$}
\psfrag{a}{\footnotesize{$\alpha$}}
\includegraphics[width=80mm]{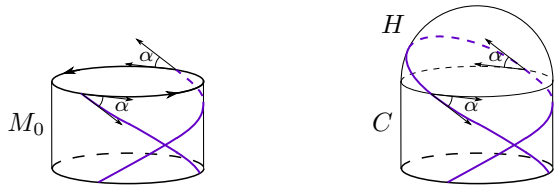}
\caption{Same scattering but not lens data.}

\end{figure}
 All known examples of nonisometric spaces with the same lens or scattering data have in common that there are trapped geodesics.

\bigskip

The scattering and lens rigidity problems are closely related to other inverse problems.  In particular, the boundary rigidity problem is equivalent to the lens rigidity question in the Simple and SGM cases.  See \cite{Cr91} and \cite{Cr04} for definitions and relations to some other problems.  There is a vast literature on these problems (see for example \cite{Be83,Bu-Iv06,Cr91,Cr90,Gr83,Mi81,Mu77,Ot90-2,Pe-Sh88,Pe-Ul05}).  In particular, it was shown in \cite{Pe-Ul05} that Simple two dimensional compact manifolds are boundary rigid (hence lens and scattering rigid).  The Simple condition however precludes trapped geodesic rays.

\bigskip

The main issue in the proofs of all the Theorems in this paper is to show that the space of trapped geodesics has measure $0$.  We will get at this by counting intersections of geodesics and applying a version of Crofton's formula.  We do this in Section \ref{intersections}.

We prove Theorems \ref{flatcylinder} and Theorem \ref{flatmobius} in section \ref{flatcase} using rigidity arguments developed in \cite{Cr91} and \cite{Cr-Kl98}.  Theorem \ref{negative curvature} is proved in section \ref{negativecase} using a rigidity method developed by Otal in \cite{Ot90-1,Ot90-2}).

\section{Counting Intersections}
\label{intersections}

In this section we discuss how to use a version of Crofton's formula to show that trapped geodesics have measure $0$.

We begin with the general case of two $2$-dimensional manifolds $M$ and $M_1$ with the same boundary and the same scattering data.

The space of geodesics that start at the boundary can be parametrised by their initial vector in $U^+\partial M$. For $s\in \partial M$ and $\theta\in [-\frac \pi 2,\frac \pi 2]$ let $\gamma_{(s,\theta)}$ be the geodesic starting at $s$ that makes an angle $\theta$ with the inward direction.  The Liouville measure on the space of geodesics leaving the boundary can be represented as $|\cos(\theta)| d\theta ds$, where $ds$ represents the arclength along the boundary.  In fact, Santal\'o's formula (see chapter 19 of \cite {Sa76}) tells us that this is true for any curve $\tau$ in $M$.  Namely, if we parametrise the geodesics passing through $\tau$ by arclength $dt$ along $\tau$ and angle $\phi$ made with a chosen normal, then the Liouville measure will be $|\cos(\phi)| d\phi dt$.  Of course $\gamma_{(s,\theta)}$ might intersect the curve $\tau$ many times (or not at all). Let $i(\tau,s,\theta)$ be the geometric number of times that $\gamma_{(s,\theta)}$ intersects $\tau$.  Also let $G(\tau(t))$ be the subset of the unit vectors at $\tau(t)$ that are tangent to geodesics that started at a boundary point.  The above gives us the following version of Crofton's formula (which works in both $M$ and $M_1$):
$$\int_{\partial M} \int_{-\frac\pi 2}^{\frac\pi 2} i(\tau,s,\theta) |\cos(\theta)| d\theta ds  = \int_\tau \int_{G(\tau(t))} |\cos(\phi)| d\phi dt.$$

We will let $\Gamma$ (resp $\Gamma_1$) be the space of non-trapped geodesics that are not tangent to the boundary at either endpoint.  $\Gamma$ can be parameterized as an open subset of the unit vectors $U^+\partial M$ on the boundary pointing inward.  $\Gamma_1$ can be identified with $\Gamma$ by this parametrization.  We will consider the corresponding intersection functions $i(\gamma,\tau)$ and $i_1(\gamma_1,\tau_1)$ which map $\Gamma\times \Gamma-Diag$ to the nonnegative integers via the geometric intersection number (i.e. the number of intersection points) of the geodesics $\gamma$ and $\tau$ (respectively $\gamma_1$ and $\tau_1$), where $\gamma$ and $\tau$ are distinct non trapped geodesics (running from boundary point to boundary point) of $M$ and $\gamma_1$ and $\tau_1$ are the corresponding geodesics in $M_1$.  We will show that these functions are closely related.  They need not be the same though as the counter example to scattering rigidity for the M\"obius strip has $i_1 = i+1$.

\begin{lemma}
Let $\gamma$, $\tau^0$ and $\tau^1$  be distinct elements of $\Gamma$ such that $\tau^0$ and $\tau^1$ are in the same component of $\Gamma$.  Then
$$i(\gamma,\tau^0)-i_1(\gamma_1,\tau_1^0)=i(\gamma,\tau^1)-i_1(\gamma_1,\tau_1^1).$$
\end{lemma}

\begin{proof} Since $\Gamma$ is an open subset of a $2$-dimensional manifold we can (by standard transversality arguments) choose a smooth path $\tau^t$ from $\tau^0$ to $\tau^1$ such that $\tau^t\not= \gamma$ and $\tau^t\not= -\gamma$ for any $t\in [0,1]$, and $\tau^t$ intersects transversely the subspace $End(\gamma)$ of $\Gamma$ consisting of geodesics with an endpoint in common with $\gamma$.  In particular, if an endpoint of $\tau^{t_0}$ (say $\tau^{t_0}(0)$) coincides with an endpoint of $\gamma$, then $W=\frac {d}{dt}|_{t_0} \tau^t(0)$ is not the zero vector.  Since geodesics always intersect transversely (except at boundary points) $f(t)=i(\gamma,\tau^t)$ (resp $f_1(t)=i_1(\gamma,\tau^t)$) only changes for those $t_0$'s when $\tau^{t_0}\in End(\gamma)$.  As we pass trough $t_0$ $f(t)$ and $f_1(t)$ change by exactly 1 (either plus or minus).  However the sign of the change is determined by $W$ (more precisely, the direction on the boundary determined by $W$) and the inward tangents to $\gamma$ and $\tau$ at the common boundary point.  That is, if the inward tangent to $\gamma$ lies between $W$ and the inward tangent to $\tau^{t_0}$, then both $f$ and $f_1$ increase by one and they will decrease by one otherwise.  In either case we see that $f(t)-f_1(t)$ is constant.
\end{proof}

We will apply this lemma to our various cases.  In the case of the flat M\"obius strip $\Gamma$ is connected and hence $i_1=i+n$ for some integer $n$.  However, there are geodesics $\gamma$ and $\tau$ in $M$ that don't intersect at all so $0\leq i_1(\gamma_1,\tau_1)=0+n$.  Hence $n$ is a nonnegative integer.  In the case of the flat torus $\Gamma$ has two components, but since one component is gotten from the other by reversing orientations of the geodesics, and since intersection numbers are independent of orientation, we again conclude $i_1=i+n$ where as before $n$ is a nonnegative integer.

Consider the case of a negatively curved cylinder with convex boundary with boundary components $\partial^1$ and $\partial^2$. It is straightforward to see that (up to reversing orientations) there are three components:  Those geodesic going from $\partial^1$ to $\partial^1$; those going from $\partial^2$ to $\partial^2$; and those going from $\partial^1$ to $\partial^2$.  However, for any pair of such components (including when both are the same component) we can find a geodesic from each component that do not intersect each other.  The previous argument then tells us that $i_1\geq i$.\

\bigskip
\bigskip

Our next goal is to study the measure of the set of trapped geodesics.  To that end, for a surface $M$ with boundary, we let $TG^+(x)\subset U_x$ be the set of unit vectors $v$ at $x\in M$ such that the geodesic ray in the $v$ direction never hits the boundary.  Further we define $TG^-= \{v|-v\in TG^+\}$, $TG(x)=TG^+(x)\cup TG^-(x)$ (the trapped directions), and $TTG(x)=TG^+(x)\cap TG^-(x)$ (the totally trapped directions).  We say the space of trapped geodesics has measure $0$ if the measure of $TG(x)$ is $0$ for all $x$.

\begin{lemma}
\label{crofton}
Let $M$ and $M_1$ be surfaces with the same scattering data and $\gamma\in \Gamma$.  Assume that the space of trapped geodesics in $M$ has measure $0$.  If for every $\tau\in \Gamma$ we have $i(\gamma,\tau)\leq i_1(\gamma_1,\tau_1)$ then $L(\gamma)\leq L(\gamma_1)$.  Further if $L(\gamma) = L(\gamma_1)$ then $TG(\gamma_1(t))$ has measure $0$ for almost all $t$.

\end{lemma}

\begin{proof}
First note that
$$4L(\gamma_1)= \int_0^{L(\gamma_1)} \int_0^{2\pi} |cos(\theta)| d\theta dt \geq \int_0^{L(\gamma_1)}\int_{G(\gamma_1(t))} |cos(\theta)|d\theta dt.$$

While Crofton's formula says
$$\int_0^{L(\gamma_1)}\int_{G(\gamma_1(t))} |cos(\theta)|d\theta dt = \int_{\Gamma} i_1(\gamma_1,\tau_1) d\tau_1\geq \int_{\Gamma} i(\gamma,\tau) d\tau=4L(\gamma).$$
In the above we used that the measures $d\tau_1$ and $d\tau$ on $\Gamma$ are the same.  In order for equality to hold not only must $i_1(\gamma,\cdot)$ and $i(\gamma,\cdot)$ coincide but $TG(\gamma_1(t))$ must have measure $0$ for almost all $t$.
\end{proof}

\section{The flat case}
\label{flatcase}

In this section we will prove Theorems \ref{flatcylinder} and \ref{flatmobius}. We will start by considering the cylinder case.  Let $M=[0,1]\times S^1$ be a flat cylinder and suppose $(M_1, \partial M_1, g_1)$ is a surface with the same lens data as $M$.

We see that the geodesics that start perpendicular to the boundary (and hence end perpendicular to the boundary) all have length $1$ and achieve the distance between the boundary components.  In particular they are minimizing geodesics, no two intersect and the union covers $M_1$ (since a shortest path from any interior point of $M_1$ to the boundary will hit the boundary perpendicularly).  Thus there is a natural diffeomorphism $F:M=\{(t,\theta)\in [0,1]\times S^1\} \to M_1$.  Along the geodesic $\gamma_{1\theta}$ of $M_1$ that starts perpendicular to the boundary at $(0,\theta)$ the vector field  $\frac d {d\theta} = j(t,\theta)N_1(t,\theta)$ (where $N_1(t)$ is the unit vector field in the $\frac d {d\theta}$ direction) is a Jacobi field perpendicular to $\gamma_{1\theta}$.  By the above
$$Area(M_1)=\int_{S^1} \int_0^1 j(t,\theta)\ dt d\theta.$$

The fact that $M_1$ has the same lens data as $M$ says that Jacobi fields along $\gamma_{1\theta}$ correspond to those along $\gamma_\theta$ in $M$ in the sense that, if some Jacobi field $J_1$ along $\gamma_{1\theta}$ has the same initial conditions (value and covariant derivative) as a Jacobi field $J$ along $\gamma_\theta$, then they also must have the same final conditions.  This being true for all Jacobi fields along $\gamma_{1\theta}$ is equivalent (see \cite{Cr91}) to
$$\int_0^1 j^{-2}(t,\theta)\ dt=\int_0^1 1\ dt=1.$$
But the convexity of $f(x)=x^{-2}$ tells us that $\int_0^1 1\ dt = \int_0^1 j^{-2}(t,\theta) dt \geq \{\int_0^1 j(t,\theta) dt\}^{-2}$ with equality if and only if $j(t,\theta)\equiv 1$.  And hence we see that
$$Area(M_1)=\int_{S^1} \int_0^1 j(t,\theta)\ dt d\theta \geq \int_{S^1} \int_0^1 1\ dt d\theta = Area(M)$$
with equality holding if and only if $j(t,\theta)\equiv 1$, i.e. $M_1$ is isometric to $M$ with the isometry being the diffeomorphism $F$ described above.  Thus we have shown

\begin{lemma}
\label{volinequality}
Let $M$ be a flat cylinder .  Then if $M_1$ is a surface with the same lens data then
$$Area(M_1)\geq Area(M)$$
with equality holding if and only if $M_1$ is isometric to $M$.
\end{lemma}

On the other hand we have shown in the previous section that the set of unit vectors in $M_1$ tangent to trapped geodesic rays has measure 0. (This is of course also true of $M$.)  Now Santal\'o's formula  and the invariance of the Liuoville measure under the geodesic flow tells us that the Liouville volume of the unit tangent bundle of $M$ (resp. $M_1$) is $\int_{U^+\partial M} L(\gamma(v)) dv$ (respectively $\int_{U^+\partial M_1} L_1(\gamma_1(v)) dv$), where the measures $dv=|\cos(\theta)| d\theta ds$ on $U^+\partial M$ and $U^+\partial M_1$ are the same. Thus the lens equivalence tells us that the unit tangent bundle of $M$ has the same measure as that of $M_1$ and hence the areas are the same (see chapter 19 of \cite{Sa76}).  Thus we conclude the isometry of $M$ and $M_1$, which completes the proof of Theorem \ref{flatcylinder}.
\bigskip
\bigskip

We now consider the M\"obius strip case.  We want to do this by passing to the orientation double cover of $M$ and $M_1$ and then apply Theorem \ref{flatcylinder}.  The only real issue in doing this is to see that $M_1$ is not orientable.  (Note that in the counterexample to scattering rigidity $M_1$ is orientable.)  The key point to note is that the argument in the previous section says that the geodesics leaving the boundary perpendicularly cannot intersect (or else they would be too long).  Thus in $M_1$ going across such a geodesic and following the boundary back to the original point reverses orientation (just as in $M$).  Thus we can pass to the two fold covers to complete the proof of Theorem \ref{flatmobius}.

\section{Negative curvature}
\label{negativecase}

In this section we will prove Theorem \ref{negative curvature}.

Fix a boundary point $x\in \partial M$ and its corresponding point $x_1\in \partial M_1$.  Let $\tau:(-\infty,\infty)\to \partial M$ be the unit speed parametrization of the boundary component with $\tau(0)=x$ (which of course goes around the boundary infinitely often).  Similarly define $\tau_1$.  We let $\gamma^t$ be the geodesic segment (varying continuously in $t$) from $x$ to $\tau(t)$.  Let $\gamma^t_1$ be the corresponding geodesic segment in $M_1$.

Our first goal is to show that there are no conjugate points along any geodesic in $M_1$.  By the convexity of the boundary, for $t$ near $0$ both $\gamma^t$ and $\gamma^t_1$ are minimizing.  In particular, for small $t$ there are no conjugate points along $\gamma^t_1$.  If any such geodesic $\gamma^t_1$ has a conjugate point let $t_0$ be the first $t$ (i.e. $|t_0|$ is the smallest) where this happens. Since $\gamma^t_1$ is a smooth variation, the conjugate pair must be the endpoints.  However, the lens equivalence would imply that the endpoints are also conjugate along $\gamma^{t_0}$, but this can't happen by the negative curvature assumption.  This covers all geodesics from this boundary component to itself.  Of course a similar argument works for geodesics with both boundary points on the other component.  In fact, since we also know that a minimizing geodesic between components in $M$ will correspond to a minimizing geodesic in $M_1$ between the components, we can use a similar continuity argument to see that there are no conjugate points along the geodesics going from one component to the other.  Now, since all geodesics leaving the boundary are limits of geodesics that hit the boundary at both endpoints, we see that all geodesics that start at the boundary have no conjugate points.

Next we want to compare geodesics in the universal covers $\tilde M$ and $\tilde M_1$ of $M$ and $M_1$.  Thus the first step is to show that $M_1$ is also a cylinder, i.e. that $\pi_1(M_1,x_1)=\Z$ and is generated by going once around the boundary curve, which we assume has length $L$. Using the homotopy $H_t=\gamma^t_1\cup -\tau{[0,t]}$ from the trivial curve, it follows that the geodesics $\gamma^{nL}_1$ are homotopic to going around the boundary $n$ times.  We also know, by the convexity of the boundary, that every homotopy class is represented by some geodesic loop at $x_1$.  Thus we need only show that none of these loops are trivial in homotopy.  However, if such a geodesic loop is contractible, then a standard minimax argument would yield a geodesic loop of index 1 which is precluded by the no conjugate points result.  This allows us to conclude that universal covers $\tilde M$ and $\tilde M_1$ also have the same lens data (with the boundaries in the universal covers identified by the covering).  In particular, it now follows that all geodesics between boundary points (and hence by taking limits all geodesics with one boundary endpoint) in $\tilde M$ and $\tilde M_1$ are minimizing.  One can tell whether two geodesics in $\tilde M$ with disjoint endpoints on the boundary intersect simply by looking at the endpoints.  The endpoints will force the intersection number mod 2 to be either 0 or 1.  Since geodesics can intersect at most once they will intersect if and only if this number is 1.  But this means that the corresponding pair of geodesics in $\tilde M_1$ will intersect if and only if they do in $\tilde M$.

We will need control (locally) on the covariant derivatives of the gradient of distance functions from boundary points.  Fix $\tilde x$ in the interior of $\tilde M_1$ with $d(\tilde x,\partial \tilde M_1)=d_0$.  Choose $\frac{d_0} 4 \geq \epsilon>0$ where $\epsilon$ is less than the injectivity radius for points $\tilde z \in B(\tilde x, \frac{d_0} 2)$.  Then, by compactness, there are uniform upper and lower bounds on the geodesic curvatures of $\partial B(\tilde z, \epsilon)$.  This implies that for any $\tilde y\in \partial \tilde M_1$ the level sets of $d(\tilde y,\cdot)$ have uniformly bounded geodesic curvature at points in $B(\tilde x,\frac {d_0} 4)$. This is true since for each point $\tilde q$ on the level set and each side of the level set there is a $B(\tilde z,\epsilon)$ lying on the given side and whose boundary is tangent to the level set at $\tilde q$.  (The two $\tilde z$'s lie on the geodesic from $\tilde y$ to $\tilde q$.)  Thus there is a neighborhood of $\tilde x$ and a number C such that for all $\tilde y\in \tilde M_1$ we have $|\nabla \nabla d(\tilde y, \cdot))|\leq C$ in $B(\tilde x,\frac {d_0} 4)$.

\begin{lemma}\label{Number of trapped directions}
Let $M$ be a cylinder of negative curvature with convex boundary.  If $M_1$ is a surface with the same lens data, then for every $x$ we have $TG^+(x)$ consists of at most two vectors. (Hence $TG^-(x)$, and $TTG(x)$ consists of at most two vectors while $TG(x)$ consists of at most 4 vectors.)
\end{lemma}

\begin{proof}
Fix an interior point $x\in M_1$.  To study the set of vectors tangent to geodesics from $x$ and hitting one of the boundary components we can look to the universal cover $\tilde M_1$ (whose boundary we now know has two connected components) and a point $\tilde x$ over $x$. For each point $\tilde y$ on $\partial \tilde M_1$ there is a geodesic arc from $\tilde x$ to $\tilde y$ (since the minimizing path is never tangent to the convex boundary).  Further this geodesic is unique, for if not two geodesics leaving $\tilde y$ would intersect again - but we have shown this doesn't happen.  Thus we get a map from $\partial \tilde M_1$ to the unit circle at $\tilde x$.  The fact that the map is continuous follows from the fact that we have no conjugate points along geodesics that leave the boundary.  Thus the unit tangents to geodesics leaving $\tilde x$ and hitting $\partial \tilde M_1$ come in two disjoint open intervals (one going to each component).

Thus $TG^+(x)$ is the complement in the unit circle of two disjoint closed intervals.  We will first see that the endpoints of these intervals vary continuously.  Consider the vectors $V_{\tilde y}(\tilde x)=-\nabla d(\tilde y,\cdot)$ which are tangent to the geodesic from $\tilde x$ to $\tilde y\in \partial \tilde M_1$.  These vector fields (as $\tilde x$ varies) are continuous and in a neighborhood of $\tilde X$ have uniformly bounded covariant derivatives by the argument in the paragraph before the Lemma.  The endpoints of the intervals will be limits of the $V_{\tilde y}(\tilde x)$ as $\tilde y$ runs off to infinity along an end of the boundary.  The control we have on the derivative tells us that the vector fields $V_{\tilde y}(\tilde x)$ will converge to a continuous vector field.

Since we know that the lengths are the same as in $M$, Lemma \ref{crofton} says that along any geodesic $\gamma$ between boundary points and for almost every $t$, $TG^+(\gamma(t))$ has measure $0$ and hence consists of two vectors.  Thus by continuity this holds for all $t$.  It is straightforward to see that such geodesics cover all of $M_1$.
\end{proof}

Note that since the totally trapped geodesics have measure $0$ they are limits of geodesics that hit the boundary so also have no conjugate points.

With these preliminaries the rest of the argument closely follows the proofs in \cite{Ot90-1}.  The assumption in that paper was that both spaces have negative curvature (and no boundary). However, the proofs only use this fact on the target space, along with the facts that geodesics intersect at most once in $\tilde{M_1}$ and if geodesics intersect in $\tilde M_1$ then corresponding geodesics intersect in $\tilde M$, but we have shown these facts above.  We now outline parts of the argument here but see \cite{Ot90-1} for more details.

Consider the space $\tilde{\Gamma}$ (resp $\tilde{\Gamma}_1$) of geodesics that are not totally trapped (i.e. trapped in both directions) in $\tilde M$ (resp. $\tilde M_1$) with its standard (Liouville) measure. The scattering data gives a $\pi_1$ invariant, measure preserving, homeomorphism $\varphi$ from $\tilde{\Gamma_1}$ to $\tilde{\Gamma}$.

Let $v\in U\tilde{M_1}$ and $\theta \in (0,\pi)$, denote by $\theta v$ a $\theta$ rotation of $v$ in the same fiber. If $v$ and $\theta v$ are not totally trapped, then $\sigma_v=\varphi(\gamma_{1v})$ and $\sigma_{\theta v}=\varphi(\gamma_{1\theta v})$ are geodesics in $\tilde{\Gamma}$ that intersect at one point. Let $\bar \theta(v, \theta)$ be the angle at which $\sigma_{\theta v}$ intersects $\sigma_{v}$. We define $\bar \theta(v,0)=0$ and $\bar \theta(v, \pi)=\pi$.

\begin{lemma}
$\bar\theta$ is continuous, and can be continuously extended to $U\tilde{M_1}\times[0,\pi]$.

\end{lemma}

\begin{proof}
We can parameterize $\tilde\Gamma$ by its initial vector in $U^+\partial \tilde M$, then by continuity of the geodesic flow we can see that the relation between pairs of geodesics in $U^+\partial\tilde M \times U^+\partial\tilde M$ and their intersection angle is continuous, where we consider the intersection of a geodesic with itself to have angle $0$ or $\pi$ depending on orientation. Since the same is true in $\tilde M_1$, the function $\bar\theta$ will be continuous when restricted to the set where neither $v$ nor $\theta v$ is a totally trapped direction. (If a geodesic doesn't have an initial point - i.e. is defined for all negative parameter values - and is not trapped, it will have an endpoint on the boundary and we can define $\bar\theta$ by reversing the orientation.)

Since $\tilde{M}$ is an infinite strip with negative curvature, there is only one totally trapped geodesic $\sigma_0$ in $\tilde M$. If $v\in U\tilde M_1$ is not totally trapped but $\theta_0 v\in TTG$, we extend $\bar \theta(v, \theta_0)$ to be the angle that $\sigma_v$ makes with $\sigma_0$.  Vectors $w$ converging to $\theta_0 v$ either are in $TG^-$ or $\gamma_{1w}$ will have basepoint in  $\partial \tilde M$ at a distance from $\gamma_{1v}(0)$ going to infinity. Therefore, $\sigma_{w}$ will have the same property and (if it converges) will converges to a geodesic in $TG^-$, by the same argument also in $TG^+$ therefore totally trapped.  Thus the $\sigma_w$ converges to $\sigma_0$, and our extension will be continuous.

If $\gamma_{1v}$ is totally trapped, we can reverse the roles of $v$ and $\theta v$. They can't be both totally trapped without being the same geodesic, since totally trapped geodesics can not intersect by Lemma \ref{Number of trapped directions}.

\end{proof}

Note that the equivariance of the metrics on the universal cover allows us to define $\bar \theta (v,\theta)$ for $v\in U M_1$ (rather than $U\tilde M_1$).

Define the \emph{average angle} as $$\Theta(\theta)=\frac{1}{Vol(UM_1)}\int_{UM_1}\bar \theta(v, \theta) dv$$
were $dv$ is the Liouville measure in $UM_1$.

\begin{prop}\label{Theta is additive}
$\Theta :[0,\pi]\to[0,\pi]$ is an increasing homeomorphism such that:
\begin{enumerate}
\item $\Theta$ is symmetric in $\pi-\theta$.
\item $\Theta$ is super-additive
\end{enumerate}
Moreover, if $\Theta$ is additive, the images under $\varphi$ of any three geodesics that intersect at a common point, also intersect at one point.
\end{prop}

In the above $(1)$ means $\Theta(\pi-\theta)=\pi-\Theta(\theta)$ while $(2)$ means $\Theta(\theta_1+\theta_2)\geq \Theta(\theta_1)+\Theta(\theta_2)$ whenever $\theta_1+\theta_2\in [0,\pi]$.
The Proposition follows directly from the proofs in \cite[Section 2]{Ot90-1}. (Note that in that paper $\theta'$ is used instead of $\bar \theta$ and $\Theta'$ instead of $\Theta$.)
\\

Let $F:[0,\pi]\to \R$ be a continuous convex function. By Jensen inequality, for each value of $\theta$
 $$F(\Theta(\theta))\leq \frac{1}{Vol(UM_1)}\int_{UM_1}F(\bar\theta(v, \theta)) dv. $$
Integrating over $[0,\pi]$ with measure $sin(\theta) d\theta$, and using Fubini we get
$$\int_0^\pi  F(\Theta(\theta))sin(\theta) d\theta \leq \frac{1}{Vol(UM_1)}\int_{UM_1} \int_0^\pi F(\bar\theta(v, \theta)) sin(\theta) d\theta dv. $$
Let $\bar F(v)=\int_0^\pi F(\bar\theta(v, \theta)) sin(\theta) d\theta$, so
$$\int_0^\pi  F(\Theta(\theta))sin(\theta) d\theta \leq \frac{1}{Vol(UM_1)}\int_{UM_1} \bar F(v) d\theta dv. $$

\begin{lemma}
Let $(M,\partial M,g)$ and $(M_1,\partial M_1,g_1)$ be as above, and $F:[0,\pi]\to \R$ any convex function. Then
$$\int_0^\pi  F(\Theta(\theta)) sin(\theta) d\theta \leq \int_0^\pi  F(\theta)sin(\theta) d\theta.$$
\end{lemma}

It suffices to prove that $$\frac{1}{Vol(UM_1)}\int_{UM_1} \bar F(v)  dv =\int_0^\pi  F(\theta)sin(\theta) d\theta.$$

For this we will first average $\bar F$ along each nontrapped geodesic $\gamma_1$. Let $\gamma=\varphi(\gamma_1)$ then $\varphi$, which is a homeomorphism when restricted to the nontrapped geodesics, induces a homeomorphism from $\gamma_1\times(0,\pi)$ to $\gamma\times(0,\pi)$ by $\Phi(\gamma_1(t),\theta)=(\gamma(\bar t),\bar\theta(\gamma_1'(t),\theta))$, where $\gamma(\bar t)$ is the point of intersection. This sends the Liouville measure $d\lambda = sin(\theta) d\theta dt$ to $d\bar \lambda = sin(\bar\theta) d\bar\theta d\bar t$.  (Note that in the earlier sections $\theta$ represented the angle from the normal to the curve where here it represents the angle from the tangent.  This is why the measure here has a $\sin(\theta)$ while before it was $|\cos(\theta)|$ ).  Therefore
$$\frac{1}{L(\gamma_1)}\int_{\gamma_1} \bar F(\gamma_1'(t)) dt =
\frac{1}{L(\gamma_1)}\int_{\gamma_1\times(0,\pi)}  F(\bar\theta(\gamma_1'(t),\theta))sin(\theta) d\theta dt $$
$$ = \frac{1}{L(\gamma_1)}\int_{\gamma\times(0,\pi)}  F(\bar\theta)sin(\bar\theta) d\bar\theta d\bar t=\frac{L(\gamma)}{L(\gamma_1)} \int_0^\pi  F(\bar\theta)sin(\bar\theta) d\bar\theta. $$
Since the lengths of $\gamma$ and $\gamma_1$ coincide, we have that
$$\frac{1}{L(\gamma_1)}\int_{\gamma_1} \bar F(\gamma_1'(t)) dt =\int_0^\pi  F(\theta)sin(\theta) d\theta$$
along each nontrapped geodesic, and since trapped directions have measure $0$, the average over $UM_1$ is the same.

\begin{lemma}(Lemma 8 from \cite{Ot90-1})
 Let $\Theta:[0,\pi]\to[0,\pi]$ be an increasing homeomorphism such that
\begin{enumerate}
 \item  $\Theta$ is super-additive and symmetric in $\pi-\theta$.
 \item for all continuous convex function $F:[0,\pi]\to \R$
$$\int_0^\pi  F(\Theta(\theta))sin(\theta) d\theta\leq \int_0^\pi  F(\theta)sin(\theta) d\theta.$$
\end{enumerate}
Then $\Theta$ is the identity.
\end{lemma}

\begin{proof}[Proof of Theorem \ref{negative curvature}]

By the previous lemma $\Theta = Id$. In particular $\Theta$ is additive, so by Lemma \ref{Theta is additive} the images under $\varphi$ of any three geodesics that intersect at a point also intersect at one point. This determines a well defined map $f:\tilde{M_1} \to \tilde{M}$ that is $\pi_1$ invariant since $\varphi$ is.

Let $\gamma_1$ be a geodesic segment from the boundary to a point $x\in M_1$, and $\gamma= f(\gamma_1)$ the corresponding segment in $M$ between $\gamma_1(0)$ and $f(x)$. Since $\Phi(\gamma_1(t),\theta)=(\gamma(\bar t),\bar\theta(\gamma_1'(t),\theta))$ sends the measure $sin(\theta) d\theta dt$ to $sin(\bar\theta) d\bar\theta d\bar t$, we get

$$L(\gamma_1) =
\frac{1}{2}\int_{\gamma_1\times(0,\pi)}  sin(\theta) d\theta dt $$
$$ = \frac{1}{2}\int_{\gamma\times(0,\pi)}  sin(\bar\theta) d\bar\theta d\bar t=L(\gamma). $$
Therefore, the lengths of geodesics segments is preserved by $f$, and so it is an isometry.

\end{proof}

\end{document}